\documentclass[a4paper, twoside, 12pt]{amsart}

\usepackage{amssymb}
\usepackage{amsmath}
\usepackage{amsthm}
\usepackage[abbrev, nobysame]{amsrefs}
\usepackage{enumitem}

\newtheorem{theorem}{Theorem}[section]
\newtheorem{lemma}[theorem]{Lemma}

\newtheorem{proposition}[theorem]{Proposition}
\newtheorem{claim}[theorem]{Claim}
\newtheorem*{claim*}{Claim}
\theoremstyle{definition}
\newtheorem{corollary}[theorem]{Corollary}
\newtheorem{definition}[theorem]{Definition}
\newtheorem{remark}[theorem]{Remark}

\newcommand{\eps}{\varepsilon}

\newcommand{\Db}{\mathbb{D}}

\newcommand{\Cb}{\mathbb{C}}
\newcommand{\Qb}{\mathbb{Q}}
\newcommand{\Zb}{\mathbb{Z}}

\newcommand{\Fl}{\mathcal{F}}

\newcommand{\Det}{\mathrm{Det}}

\begin{document}

\title[Bounded multiplicity of eigenvalues]{Bounded multiplicity for eigenvalues of a circular vibrating clamped plate}
\author{Yuri Lvovsky and Dan Mangoubi}
\address{The Einstein Institute of Mathematics, Edmond J. Safra Campus,
	 The
	Hebrew University of Jerusalem, Jerusalem 91904, Israel}
\email[YL]{yuri.lvovsky@mail.huji.ac.il}
\email[DM]{dan.mangoubi@mail.huji.ac.il}

\begin{abstract}
	We prove that no eigenvalue of the clamped disk can have multiplicity
	greater than six. Our method of proof is based on a new recursion formula, 
	linear algebra arguments and a transcendency theorem due to Siegel and Shidlovskii.
\end{abstract}
\maketitle
%================================================
\section{Introduction and background}
%================================================
\subsection{The vibrating membrane}
%--------------------------------
Recall the Dirichlet eigenvalue problem on the unit disk, $\Db$.
$$(\mathrm{VM})\quad\left\{\begin{matrix}
-\Delta u &=& \lambda u & \text{ in } \Db,\\
u&=&0 &\text{ on }\partial\Db,
\end{matrix}\right.
$$
where $\Delta=\mathrm{div}\circ{\mathrm{grad}}$ is the (analyst's) Laplacian.
The eigenfunctions and the corresponding eigenvalues 
are given in terms of Bessel functions of the first kind~$J_m$  and their positive zeros~$j_{m,k}$.
Indeed, it is straightforward to check that
\begin{equation} 
  \label{eqn:vm-basis}
  u_{m, k}(r,\phi):=J_m(j_{m,k}r)e^{im\phi}
\end{equation}
is an eigenfunction
of eigenvalue~$\lambda=j_{m, k}^2$. Fourier expansion shows that any eigenfunction
is a linear combination of functions~$u_{m,k}$ 
\cite{courant-hilbert-I}*{Ch.\ V\S5}.
On the other hand, to determine
which linear combinations of the basic eigenfunctions in~\eqref{eqn:vm-basis} still remain eigenfunctions had been a difficult problem,  until it was resolved by Siegel~\cite{sieg-1929} in his
celebrated theorem showing that the multiplicity of  the eigenvalue $j_{m,k}^2$ is either  
 one (in case $m=0$) or two (in case $m\neq 0$). 
This was coined as Bourget's Hypothesis before Siegel's Theorem.

We recall the line of proof of Bourget's hypothesis.
First, (see~\cite{wats}*{Ch.\ 15.28}) using a well known (length two) recursion
formula for Bessel functions and their  second order ODEs 
it was shown that if $j_{m, k}=j_{m', k'}$, then either $m=m'$ or
$j_{m, k}$ is algebraic.
In a second much deeper step it was shown by Siegel \cite{sieg-1929} (see also~\cite{sieg-1949}) that
all positive zeros of Bessel functions are transcendental.

%-------------------------------
\subsection{The vibrating clamped plate}
%-------------------------------
In this paper we are interested in the vibrating clamped circular plate~(\cite{courant-hilbert-I}*{Ch.\ V\S6}). This is the following fourth order eigenvalue problem.
$$\mathrm{(VP)}\quad\left\{\begin{matrix}
\Delta^2 u &=& \lambda u & \text{ in } \Db,\\
u &=& 0 & \text{ on } \partial\Db ,\\
\partial_n u &=& 0&\text{ on }\partial\Db .
\end{matrix}\right.
$$

Similarly to the vibrating circular membrane, it is readily checked that 
$$u_{m, k}(r, \phi)=\big(I_m(w_{m,k})J_m(w_{m,k}r)-J_m(w_{m,k})I_m(w_{m,k}r)\big)e^{im\phi}$$ is an eigenfunction of eigenvalue $\lambda=w_{m, k}^4$,  where~$w_{m,k}$ is a zero
of the cross product
\begin{equation}
\label{eqn:Wm-def}
  W_m:=I_m'J_m-I_m J_m'\ ,
\end{equation}
and where $I_m$ is the modified Bessel function.

As~in Problem~(VM), it is natural to ask whether multiplicities occur.
There is extensive literature studying the vibrating clamped plate
problem in general domains. The main questions studied are the isoperimetric problem,
eigenvalues inequalities, asymptotic distribution of eigenvalues and the positivity of the ground state
(see e.g.~\citelist{\cite{nad95}\cite{talenti}\cite{ashb-beng}\cite{ashbaugh-laugesen}\cite{levine-protter}\cite{cheng-wei}\cite{pleijel-plates}\cite{payne-polya-weinb}\cite{garabedian-conformal}\cite{duffin}\cite{hedenmalm}\cite{hadamard}}).
It seems that the question of multiplicity of eigenvalues for the circular plate has not been addressed so far, and it is still not known
whether eigenvalues
are of multiplicity at most two (see in this context Theorem~\ref{thm:w0-w1-zero-forbidden}).
From Weyl's law \cite{courant-hilbert-I}*{Ch.\ VI\S7.4} readily it follows
that the multiplicity of the~$k$-th eigenvalue $m(\lambda_k)=o(k)$ as $k\to\infty$.
In this paper we follow the line of proof for the bounded multiplicity of
the eigenvalues of the vibrating membrane, and we adapt it to deal
with the eigenvalues of the clamped plate problem.
The main new ingredient is a recursion formula for the sequence of cross products~$W_m$. Although this sequence was extensively studied~\cite{lorch-monotonicity}
we could not find this recursion in the existing literature.
 Further, it turns out that this recursion (of length four)  has nice grading and non-cancellation properties which  allow
to adjust the linear algebra and ODE arguments
in the proof for the vibrating membrane case to our case. When combined with  Siegel-Shidlovskii Theory
(see~\cite{shid-book}) it yields
 \begin{theorem}
 	\label{thm:wm-zeros}
 	Let $m_0, m_1, m_2, m_3$ be four distinct non-negative integers.
 	 There is no~$x_0>0$ for which
 	 $W_{m_0}(x_0)=W_{m_1}(x_0)=W_{m_2}(x_0)=W_{m_3}(x_0)=0$.
 \end{theorem}
As a main corollary we obtain
\begin{corollary}
	Let $\lambda$ be an eigenvalue of Problem $\mathrm{(VP)}$.
	Then, $\lambda$ is of  multiplicity at most six.
\end{corollary}

\begin{remark}
	One can check that the ground state of the disk is of multiplicity one (see \cite{lorch-monotonicity}). 
\end{remark}

%In the last section we also show how it follows
%that any distinct four~$W_m$s must be algebraically independent functions over~$\Cb(z)$.
 
%------------------------------
\subsection{Acknowledgements}
%------------------------------
We first learned about the clamped plate problem from Iosif Polterovich.
We are very grateful to Iosif for introducing us to this problem, and explaining to us the beauty and subtle points of several surrounding questions.
We would like to thank Enrico Bombieri, who explained to us some ideas of transcendental number theory.
This manuscript also benefited from several interesting discussions with 
Lev Buhovski, Aleksandr Logunov, Eugenia Malinnikova, Guillaume Roy-Fortin,  and Mikhail Sodin.
This paper is part of the PhD thesis of the first author.
The support of the Israel Science Foundation through grants nos.\ 753/14 and 681/18 is gratefully acknowledged.
Part of this work was written while the second author was an invited
researcher of the LabEx Math\'ematiques Hadamard project in Paris-Sud XI and a Chateaubriand France-Israel fellow.
The financial supports of the LMH and the french government
are gratefully acknowledged.

%===============================================
\section{Classical facts about Bessel functions}
%================================================
Let~$m$ be a non-negative integer.
The Bessel function~$J_m$ can be defined as the  entire function
satisfying
$$J_m''(z)+\frac{1}{z}J_m'(z)-\left(\frac{m^2}{z^2}-1\right)J_m(z)=0$$
and normalized by 
$$J_m(z)=\frac{1}{m!}\left(\frac{z}{2}\right)^m + o(z^m) \mbox{ as } z\to 0\ .$$
The modified Bessel function~$I_m$ is the entire solution of
$$I_m''(z)+\frac{1}{z}I_m'(z)-\left(\frac{m^2}{z^2}+1\right)I_m(z)=0$$
normalized so that
$$I_m(z)=\frac{1}{m!}\left(\frac{z}{2}\right)^m+o(z^m) \mbox{ as } z\to 0\ .$$

It is easily verified that $I_m(z)=(-i)^m J_m(i z)$, where $i$~is a square root of~$-1$.
\begin{proposition}[\cite{wats}*{Ch.\ II.12}]
	\label{prop:bessel-recursions}
	The Bessel functions satisfy the following rules:
	\begin{equation}
	\begin{aligned}
	J_m'(z)-\frac{m}{z}J_m(z)&=-J_{m+1}\\
	J_m'(z)+\frac{m}{z}J_m(z)&=J_{m-1}\\
	I_m'(z)-\frac{m}{z}I_m(z)&=I_{m+1}\\
	I_m'(z)+\frac{m}{z}I_m(z)&=I_{m-1}
	\end{aligned}
	\end{equation}
\end{proposition}
The next positivity statement (which can be readily seen from the Taylor series expansion) will also be useful.
\begin{proposition}
	The function~$I_m$ is positive in~$(0,\infty)$.
\end{proposition}

%====================================================================
\section{A recursion formula for a cross product of Bessel functions}
%====================================================================
As explained in the introduction, the eigenvalues of the clamped plate problem
are given in terms of zeros of the functions~$W_m$ defined in~\eqref{eqn:Wm-def}.
In this section we study this sequence and we present a length four
rational recursion relation satisfied by it.
We prove
\begin{theorem}
	\label{thm:recursion}
	The following  recursion formula holds.
	%	\begin{equation*}
	\begin{equation*}
	\begin{aligned}
	W_{m+2}(z)+W_{m+4}(z)  =& \frac{4(m+2)(m+3)}{z^2}
	\big(W_{m+1}(z)-W_{m+3}(z)\big)+\\&
	-\frac{m+3}{m+1}\big(W_{m}(z)+W_{m+2}(z)\big)\ .
	\end{aligned}
	\end{equation*}
\end{theorem}

For the proof we need some convenient formulas given in the next lemma,
proved at the end of this section.
\begin{lemma}
	\label{lem:Wm-formulas}
	The following formulas hold.
	\begin{enumerate}[label=\textup{(\alph*)}]
	  \item \ \ $W_m = \ I_{m+1}J_m+I_m J_{m+1}$\label{eq:wm+}
	  \item \ \ $W_m = \ I_{m-1} J_m-I_m J_{m-1}$\label{eq:wm-}
	  \item \ \ $W_{m-1}(z)+W_{m+1}(z) = \ \frac{4m}{z}I_m(z) J_m(z)$\label{eq:w+w}
	\item \ \ $W_{m-1}-W_{m+1} = \ 2(I_m J_m)'$\label{eq:w-w}
	\item \ \ $W_{m}+W_{m+1} = \ 2I_m J_{m+1}$\label{eq:w0+w1}
	\item \ \ $W_{m}-W_{m+1} = \  2I_{m+1}J_m$\label{eq:w0-w1}	
	\end{enumerate}
%	\begin{alignat}{2}
%	&W_m &= &\ I_{m+1}J_m+I_m J_{m+1}\label{eq:wm+}\\
%	&W_m &= &\ I_{m-1} J_m-I_m J_{m-1}\label{eq:wm-}\\
%	&W_{m-1}(z)+W_{m+1}(z)& = &\ \frac{4m}{z}I_m(z) J_m(z)\label{eq:w+w}\\
%	&W_{m-1}-W_{m+1}& = &\ 2(I_m J_m)'\label{eq:w-w}\\
%	&W_{m}+W_{m+1}& = &\ 2I_m J_{m+1}\label{eq:w0+w1}\\
%	&W_{m}-W_{m+1}& = &\  2I_{m+1}J_m\label{eq:w0-w1}
%	\end{alignat}
\end{lemma}
\begin{proof}[Proof of Theorem~\ref{thm:recursion}]
	For convenience  we denote the formula to be proved as $A=B-C$ where $A$~is the left hand side and $B, C$ correspond respectively to the two terms in the right hand side.
	 By Lemma~\ref{lem:Wm-formulas} we have
	 \begin{equation*}
	 \begin{aligned}
	  A&=\frac{4(m+3)}{z}I_{m+3}(z)J_{m+3}(z)\ ,\\
	  B&=\frac{4(m+2)(m+3)}{z^2}2(I_{m+2}J_{m+2})'(z)\\
	  C&=4\frac{m+3}{z}I_{m+1}(z)J_{m+1}(z)\ .
	  \end{aligned}
	  \end{equation*}
	Hence, the statement $A+C=B$ is equivalent to
	\begin{equation*}
	I_{m+1}(z)J_{m+1}(z)+I_{m-1}(z)J_{m-1}(z)
	=\frac{2m}{z}(I_{m}J_{m})'(z)\ .
	\end{equation*}
	The last identity can be easily validated by expressing $I_{m+1}$,~$I_{m-1}$,~$J_{m+1}$ and~$J_{m-1}$ in terms of~$I_m'$,~$I_m$,~$J_m'$ and~$J_m$ with the rules given in Proposition~\ref{prop:bessel-recursions}.
\end{proof}

\begin{proof}[Proof of Lemma~\ref{lem:Wm-formulas}]
	To prove~\ref{eq:wm+} we use the rules in Proposition~\ref{prop:bessel-recursions}
	to obtain
	\begin{equation*}
	\begin{aligned}
	I_{m+1}(z)&J_m(z)+I_m(z)J_{m+1}(z)=\\
	& \left(  I_m'(z)-\frac{m}{z}I_m(z) \right) J_m(z)-
	I_m(z)\left(J_{m}'(z)-\frac{m}{z}J_m(z)\right)=\\
	& \ \  I_m'(z)J_m(z)-I_m(z)J_m'(z)=W_m(z)\ .
	\end{aligned}
	\end{equation*}
	Formula~\ref{eq:wm-} is proved similarly.
	To prove~\ref{eq:w+w} we express $W_{m-1}$ using formula~\ref{eq:wm+},
	while~$W_{m+1}$ using formula~\ref{eq:wm-}. Then, we get
	$$W_{m-1}+W_{m+1}=I_m(J_{m-1}+J_{m+1})+(I_{m-1}-I_{m+1})J_m\ .$$
	At the next step we express $J_{m-1}$,~$J_{m+1}$,~$I_{m-1}$ and~$I_{m+1}$
	in terms of the functions $J_m'$,~$J_m$,~$I_m'$ and~$I_m$ using Proposition~\ref{prop:bessel-recursions}.
	The proof of~\ref{eq:w-w} is similar.
	To prove~\ref{eq:w0+w1} and~\ref{eq:w0-w1} one uses~\ref{eq:wm+} and~\ref{eq:wm-}.
\end{proof}

%==============================
\section{Forbidden joint zeros}
%==============================
In this section we observe some forbidden patterns of joint zeros
in the sequence $W_m$. Observe that the forbidden patterns in Theorem~\ref{thm:w0-w1-zero-forbidden} are not covered
by Theorem~\ref{thm:wm-zeros}.
\begin{theorem}
	\label{thm:w0-w1-zero-forbidden}
	The patterns of joint zeros below are forbidden.
	\begin{enumerate}[label=\textup{({\alph*})}]
	\item \label{thm-part-a} The functions $W_m$ and $W_{m+1}$ have no joint positive zeros.
	\item \label{thm-part-b} The functions $W_m$ and $W_{m+2}$ have no joint positive zeros.
	\end{enumerate}
 	\begin{proof}
 	  Since~$I_m$ is a positive function in~$(0,\infty)$, we can deduce from Lemma~\ref{lem:Wm-formulas}\ref{eq:wm+} and~\ref{eq:wm-} that if
 		$W_m(x_0)=W_{m+1}(x_0)=0$ then $J_m(x_0)=J_{m+1}(x_0)=0$.
 		This is impossible as it implies $J_m(x_0)=J_m'(x_0)=0$, which is forbidden
 		by the second order ODE satisfied by~$J_m$.
 		To prove~\ref{thm-part-b} we use Lemma~\ref{lem:Wm-formulas}\ref{eq:w+w} and~\ref{eq:w-w} and the fact that $I_m$
 		is a positive function in~$(0,\infty)$ to deduce a similar contradiction.
 	\end{proof}
\end{theorem}

%==================================
\section{A joint zero is algebraic}
%==================================
In this section we show that the recursion given in Theorem~\ref{thm:recursion} combined with the fact that the four functions~$W_0, W_1, W_2, W_3$
do not have a joint positive zero (as follows from Theorem~\ref{thm:w0-w1-zero-forbidden}) implies that a joint zero of four distinct $W_m$'s
must be algebraic. We emphasize that this implication is independent of the specific nature of functions~$W_m$ (for example, it does not depend
on the non-trivial fact that the $W_m s$ are linearly independent - see Appendix).

\begin{proposition}
	\label{prop:algebraic}
Let $\Fl$ be a linear subspace of meromorphic functions in~$\Cb$. Let $(f_m)_{m=0}^{\infty}$ be any sequence in~$\Fl$ which satisfies the recursion relation given in Theorem~\ref{thm:recursion} and 
assume that $f_0, f_1, f_2$ and $f_3$ have no common positive zero.
Let $m_0, m_1, m_2, m_3$ be distinct  non-negative integers.
Let $x_0>0$ be such that 
 $f_{m_0}(x_0)=f_{m_1}(x_0)=f_{m_2}(x_0)=f_{m_3}(x_0)=0$.
 Then $x_0$ is algebraic.
\end{proposition}
The heart of the proof of Proposition~\ref{prop:algebraic} is
a linear independence property implied by the recursion of Theorem~\ref{thm:recursion}.
\begin{lemma}
	\label{lem:rec-ind}
	Let~$V$ be a four dimensional linear space over the field of rational functions with rational coefficients~$\Qb(z)$.
	Let $(F_0, F_1, F_2, F_3)$ be a basis of~$V$, and define a sequence
	$(F_m)_{m=0}^{\infty}$ in~$V$ by the recursion of Theorem~\ref{thm:recursion}. Let $m_0, m_1, m_2, m_3$ be distinct non-negative integers. Then, the set of vectors $\{F_{m_0}, F_{m_1}, F_{m_2}, F_{m_3}\}$ is linearly independent.
\end{lemma}
The proof of Lemma~\ref{lem:rec-ind} is based on
nice grading and non-cancellation properties of the recursion in Theorem~\ref{thm:recursion}. We give its proof
 below the proof of Proposition~\ref{prop:algebraic}.
\begin{proof}[Proof of Proposition~\ref{prop:algebraic}]
	Consider a space~$V$ and a sequence~$(F_m)_{m=0}^{\infty}$ as in Lemma~\ref{lem:rec-ind}. According to Lemma~\ref{lem:rec-ind}
		we can uniquely express
		$$F_{m_j}=\sum_{j=0}^{3} A_{jk} F_j\ ,$$
		where $A=(A_{jk})\in M_4(\Qb(z))$ is an invertible matrix.
		Since the sequence~$(f_m)$ satisfies the same recursion we conclude
	 that (not necessarily uniquely)
	 \begin{equation}
	 \label{eq:f_m-expansion}
	  f_{m_j}=A_{j0} f_0+A_{j1} f_1 + A_{j2} f_2 + A_{j3} f_3\ .
	 \end{equation}
	 Taking a least common denominator~$D\in \Qb[z]$ 
	 for all~$A_{jk}$s we get
	 \begin{equation*}
	 \label{eq:f_m-poly-expansion}
	 D f_{m_j}=\tilde{A}_{j0} f_0 + \tilde{A}_{j1} f_1 +\tilde{A}_{j2} f_2
	 +\tilde{A}_{j3} f_3\ ,
	 \end{equation*}
	 where $D$ and~$\tilde{A}_{jk}$ are polynomials in~$\Qb[z]$.
	Evaluation of this identity at the point~$x_0$ results in
	$$D(x_0)
	\begin{pmatrix}
		f_{m_0}(x_0)\\
		f_{m_1}(x_0)\\
		f_{m_2}(x_0)\\
		f_{m_3}(x_0)
	\end{pmatrix}=\tilde{A}(x_0)\begin{pmatrix}
	f_{0}(x_0)\\
	f_{1}(x_0)\\
	f_{2}(x_0)\\
	f_{3}(x_0)
	\end{pmatrix}$$

	The left hand side is the zero vector by our assumption, while the vector $(f_0(x_0), f_1(x_0), f_2(x_0), f_3(x_0))\in \Cb^4$ is
	not zero by our assumption. We conclude that $\tilde{A}(x_0)\in M_4(\Qb)$ is
	non-invertible. Hence, $\Det (\tilde{A})(x_0)=\Det (\tilde{A}(x_0))=0$, and since $A\in M_4(\Qb(z))$ is invertible, $\Det(\tilde A)$ is a non-zero polynomial in~$\Qb[z]$ and we can conclude that $x_0$ is algebraic.
\end{proof}
\begin{proof}[Proof of Lemma~\ref{lem:rec-ind}]
Assume that $0\leq m_0<m_1<m_2<m_3$ and  define the parameters $(j, k, l, m)$
by 
\begin{equation*}
m_0=j, \quad m_1=1+j+k,\quad  m_2=2+j+k+l,\quad m_3=3+j+k+l+m.
\end{equation*}
Let us refine the statement in the Lemma.
Consider the unique anti-symmetric four-linear form defined on~$V$ for which
$(F_0, F_1, F_2, F_3):=1$.
We need to show that $(F_{m_0}, F_{m_1}, F_{m_2}, F_{m_3})\neq 0$.
Keeping track of the leading term in these determinant-like expressions
we prove
\begin{claim*}
	There exist constants $B_{jklm}>0$ such that
	\begin{multline*}
	(F_j, F_{1+j+k}, F_{2+j+k+l}, F_{3+j+k+l+m}) =\\ (-1)^m B_{jklm} z^{-2(k+2\lfloor l/2\rfloor+m)}+P_{jklm}(z^{-2})\ ,
	\end{multline*}
	where $P_{jklm}\in\Qb[z]$ is of degree smaller than $k+2\lfloor l/2 \rfloor+m-1$.
\end{claim*} 
	The proof of the preceding claim is by induction on $j+k+l+m$. The base case  $(j, k, l , m)=(0, 0, 0, 0)$ is trivial. For the sake of shortly written expressions we introduce
	some notations to expressions appearing as coefficients in the recursion of
	Theorem~\ref{thm:recursion}.
	\begin{equation*}
	\begin{aligned}
	\alpha_{jklm}&:=4(1+j+k+l+m)(2+j+k+l+m),\\ \beta_{jklm}&:=\frac{2+j+k+l+m}{j+k+l+m},\\
	\gamma_{jklm}&:=\beta_{jklm}+1.
	\end{aligned}
	\end{equation*}
	We now unroll the determinant by applying the recursion given in Theorem~\ref{thm:recursion}.
	\begin{equation*}
	\begin{split}
	(F_j, F_{1+j+k},& F_{2+j+k+l}, F_{3+j+k+l+m}) = \\
	&(F_j, F_{1+j+k}, F_{2+j+k+l},     F_{1+j+k+l+m}+F_{3+j+k+l+m})+\\
	-&(F_j, F_{1+j+k}, F_{2+j+k+l}, F_{1+j+k+l+m}) = \\
	\alpha_{jklm}z^{-2}&(F_j , F_{1+j+k}, F_{2+j+k+l}, F_{j+k+l+m}-F_{2+j+k+l+m})+\\
	-\beta_{jklm}&(F_j , F_{1+j+k},  F_{2+j+k+l}, F_{-1+j+k+l+m}+F_{1+j+k+l+m})+\\
	-&(F_j, F_{1+j+k},  F_{2+j+k+l}, F_{1+j+k+l+m})\ .
	\end{split}
	\end{equation*}
	After a slight rearrangement we obtain
	\begin{equation}
	\label{eq:det}
	\begin{split}
	(F_j, F_{1+j+k},& F_{2+j+k+l}, F_{3+j+k+l+m})=\\
	\alpha_{jklm}z^{-2}&(F_j , F_{1+j+k}, F_{2+j+k+l}, F_{3+j+k+l+(m-3)})+\\
	-\alpha_{jklm}z^{-2}& (F_j , F_{1+j+k}, F_{2+j+k+l},
	F_{3+j+k+l+(m-1)})+\\
	-\beta_{jklm}&(F_{j}, F_{1+j+k}, F_{2+j+k+l}, F_{3+j+k+l+(m-4)})+\\
	-\gamma_{jklm}&(F_{j}, F_{1+j+k}, F_{2+j+k+l}, F_{3+j+k+l+(m-2)})\ .
	\end{split}
	\end{equation}
	We denote the expression obtained in~\eqref{eq:det} by~$X$.
	In order to apply the induction hypothesis, we distinguish several cases:
	%----------------------------------------------------------------------
	\begin{enumerate}[label=\textbf{Case \arabic*:}]
		\item  $m\geq 4$.
		In this case one gets
		by the induction hypothesis that for some polynomial~$\tilde{P}_{jklm}$ (of controlled degree)
		\begin{equation*}
		\begin{split}
		X=-(-1)^m\alpha_{jklm}B_{jkl(m-3)}&z^{-2k-4\lfloor l/2\rfloor-2(m-3)-2}+\\
		(-1)^m\alpha_{jklm}B_{jkl(m-1)} &z^{-2k-4\lfloor l/2\rfloor-2(m-1)-2}+\\
		-(-1)^m\beta_{jklm}B_{jkl(m-4)} &z^{-2k-4\lfloor l/2\rfloor-2(m-4)}+\\
		(-1)^m\gamma_{jklm}B_{jkl(m-2)} &z^{-2k-4\lfloor l/2\rfloor-2(m-2)}+
		\tilde{P}_{jklm}(z^{-2})=\\
		(-1)^m \alpha_{jklm}B_{jkl(m-1)}&z^{-2k-4\lfloor l/2 \rfloor-2m}+P_{jklm}(z^{-2}) 
		\end{split}
		\end{equation*}
		where $P_{jklm}$ is a polynomial of degree smaller than $k+2\lfloor l/2\rfloor + m-2+1$.
		\vspace{1ex}
		%-----------------------------------------------------------------------
		\item $m=3$. In this case the anti-symmetry of the determinant is used to get
		\begin{equation*}
		\begin{split}
		X=\alpha_{jklm}z^{-2}&(F_j, F_{1+j+k}, F_{2+j+k+l}, F_{3+j+k+l+0})+\\
		-\alpha_{jklm}z^{-2}&(F_j, F_{1+j+k}, F_{2+j+k+l}, F_{3+j+k+l+2})+\\
		-\gamma_{jklm}&(F_j, F_{1+j+k}, F_{2+j+k+l}, F_{3+j+k+l+1})
		\end{split}
		\end{equation*}
		By induction we have
		\begin{equation*}
		\begin{split}
		X=\alpha_{jklm}B_{jkl0}&z^{-2k-4\lfloor l/2 \rfloor-2}+\\
		-(-1)^2\alpha_{jklm}B_{jkl2}&z^{-2k-4\lfloor l/2 \rfloor-4-2}+\\
		-(-1)^1\gamma_{jklm}B_{jkl1}&z^{-2k-4\lfloor l/2 \rfloor-2}+\tilde{P}_{jklm}(z^{-2})=\\
		(-1)^m \alpha_{jklm}B_{jkl2}&z^{-2k-4\lfloor l/2 \rfloor-2m}+P_{jkl3}(z^{-2})
		\end{split}
		\end{equation*}
		where $P_{jklm}$ is of degree smaller than~$k+2\lfloor l/2\rfloor+2$.
		\vspace{1ex}
		%-----------------------------------------------------------------------
		\item $m=2, l\geq 1$.
		\begin{equation*}
		\begin{split}
		X=-\alpha_{jklm}z^{-2}&(F_j, F_{1+j+k}, F_{2+j+k+l}, F_{3+j+k+l+1})+\\
		\beta_{jklm}&(F_j, F_{1+j+k}, F_{2+j+k+(l-1)}, F_{3+j+k+(l-1)+0})+\\
		-\gamma_{jklm}&(F_j, F_{1+j+k}, F_{2+j+k+l}, F_{3+j+k+l+0})\ .
		\end{split}
		\end{equation*}
		By induction,
		\begin{equation*}
		\begin{split}
		X=-(-1)\alpha_{jklm}B_{jkl1}&z^{-2k-4\lfloor l/2 \rfloor-2-2}+\\
		\beta_{jklm}B_{jk(l-1)0}&z^{-2k-4\lfloor (l-1)/2 \rfloor}+\\
		-\gamma_{jklm}B_{jkl0}&z^{-2k-4\lfloor l/2 \rfloor}+\tilde{P}_{jklm}(z^{-2})=\\
		(-1)^m \alpha_{jklm}B_{jkl1}&z^{-2k-4\lfloor l/2 \rfloor-2m}+{P_{jkl2}(z^{-2})}
		\end{split}
		\end{equation*}
		where $P_{jkl2}$ is of degree smaller than~$k+2\lfloor l/2\rfloor+1$.
		\vspace{1ex}
		%--------------------------------------------------------------------		
		\item $m=2, l=0$.
		\begin{equation*}
		\begin{split}
		X=-\alpha_{jklm}z^{-2}&(F_j, F_{1+j+k}, F_{2+j+k+0}, F_{3+j+k+0+1})+\\
		-\gamma_{jklm}&(F_j, F_{1+j+k}, F_{2+j+k+0}, F_{3+j+k+0+0})
		\end{split}
		\end{equation*}
		where by induction
		\begin{equation*}
		\begin{split}
		X=\alpha_{jklm}B_{jk01}&z^{-2k-2-2}+\\
		-\gamma_{jklm}B_{jk00}&z^{-2k}+\tilde{P}_{jklm}(z^{-2})=\\
		(-1)^m \alpha_{jklm}B_{jk01} &z^{-2k-4\lfloor l/2 \rfloor -2m} +P_{jk02}(z^{-2})
		\end{split}
		\end{equation*}
		and $P_{jk02}$ is of degree smaller than~$k+1$.
		\vspace{1ex}
		%	-----------------------------------------------------------------
		\item $m=1, l\geq 2$.
		\begin{equation*}
		\begin{split}
		X=-\alpha_{jklm}z^{-2}&(F_j, F_{1+j+k}, F_{2+j+k+(l-1)}, F_{3+j+k+(l-1)+0})+\\
		-\alpha_{jklm}z^{-2}&(F_j, F_{1+j+k}, F_{2+j+k+l}, F_{3+j+k+l+0})+\\
		\beta_{jklm}&(F_j, F_{1+j+k}, F_{2+j+k+(l-2)}, F_{3+j+k+(l-2)+1})
		\end{split}
		\end{equation*}
		
		Hence, by hypothesis
		\begin{equation*}
		\begin{split}
		X=-\alpha_{jklm}B_{jk(l-1)0}&z^{-2k-4\lfloor (l-1)/2 \rfloor-2}+\\
		-\alpha_{jklm}B_{jkl0}&z^{-2k-4\lfloor l/2 \rfloor-2}+\\-\beta_{jklm}B_{jk(l-2)1}&z^{-2k-4\lfloor (l-2)/2 \rfloor -2}+\tilde{P}_{jklm}(z^{-2})\ .
		\end{split}
		\end{equation*}
		Now it becomes a bit trickier to tell which the leading term is.
		If~$l$ is even then it is the second one,
		so we take  $B_{jkl1}=\alpha_{jklm}B_{jkl0}$.
		If~$l$ is odd then the first two terms contribute to the leading term and are of the same  sign, so we take $B_{jkl1}=\alpha_{jklm}(B_{jk(l-1)0}+B_{jkl0})$.
		In any case, we obtain 
		$$X=(-1)^m B_{jklm}z^{-2k-4\lfloor l/2 \rfloor-2m}+P_{jkl1}(z^{-2})\ ,$$ where $P_{jklm}$ is of degree smaller than~$k+2\lfloor l/2 \rfloor$.
		\vspace{1ex}
		%-------------------------------------------------------------------
		\item $m=1, l=1$.
		\begin{equation*}
		\begin{split}
		X=-\alpha_{jklm}z^{-2}&(F_j, F_{1+j+k}, F_{2+j+k+0}, F_{3+j+k+0+0})+\\
		-\alpha_{jklm}z^{-2}&(F_j, F_{1+j+k}, F_{2+j+k+1},F_{3+j+k+1+0})\ .
		\end{split}
		\end{equation*}
		The induction gives
		\begin{equation*}
		\begin{split}
		X=-\alpha_{jklm}B_{jk00} & z^{-2k-2}+\\
		-\alpha_{jklm}B_{jk10}&z^{-2k-2}+P_{jklm}(z^{-2})=\\
		(-1)^m \alpha_{jklm}(B_{jk00}+B_{jk10})&z^{-2k-4\lfloor l/2 \rfloor -2m}+P_{jk11}(z^{-2})	
		\end{split}
		\end{equation*}  
		where $P_{jk11}$ is of degree smaller than~$k$.
		\vspace{1ex}
		%------------------------------------------------------------------------
		\item $m=1, l=0, k\geq 1$.
		\begin{equation*}
		\begin{split}
		X=-\alpha_{jklm}z^{-2}&(F_j, F_{1+j+k}, F_{2+j+k+0}, F_{3+j+k+0+0})+\\
		-\beta_{jklm}&(F_j, F_{1+j+(k-1)}, F_{2+j+(k-1)+0}, F_{3+j+(k-1)+0+0})
		\end{split}
		\end{equation*}
		leading to
		\begin{equation*}
		\begin{split}
		X=-\alpha_{jklm}B_{jk00}&z^{-2k-2}+\\
		-\beta_{jklm}&z^{-2(k-1)}+\tilde{P}_{jklm}(z^{-2})=\\
		(-1)^m \alpha_{jklm}B_{jk00}&z^{-2k-4\lfloor l/2 \rfloor-2m}+P_{jk10}(z^{-2})
		\end{split}
		\end{equation*}
		where $P_{jk10}$ is of degree smaller than~$k$.
		\vspace{1ex}
		%---------------------------------------------------------------------
		\item $m=1, l=0, k=0$.
		\begin{equation*}
		X=-\alpha_{jklm} z^{-2} (F_j, F_{1+j+0}, F_{2+j+0+0}, F_{3+j+0+0+0})
		\end{equation*}
		This simple expression gives by our hypothesis
		\begin{equation*}
		\begin{split}
		X= -\alpha_{jklm}B_{j000}&z^{-2}=\\
		(-1)^m \alpha_{jklm}B_{j000}&z^{-2k-4\lfloor l/2 \rfloor-2m}\ .
		\end{split}
		\end{equation*}
		%	\vspace{1ex}
		%-------------------------------------------------------------------
		\item $m=0, l\geq 3$.
		\begin{equation*}
		\begin{split}
		X=-\alpha_{jklm}z^{-2}&(F_j, F_{1+j+k}, F_{2+j+k+(l-2)}, F_{3+j+k+(l-2)+1})+\\
		\beta_{jklm}&(F_j, F_{1+j+k}, F_{2+j+k+(l-3)}, F_{3+j+k+(l-3)+2})+\\
		\gamma_{jklm} &(F_j, F_{1+j+k}, F_{2+j+k+(l-1)}, F_{3+j+k+(l-1)+0})\ .
		\end{split}
		\end{equation*}
		We are led to the tricky expression
		\begin{equation*}
		\begin{split}
		X=\alpha_{jklm}B_{jk(l-2)1}&z^{-2k-4\lfloor (l-2)/2 \rfloor-2-2}+\\
		\beta_{jklm}B_{jk(l-3)2}&z^{-2k-4\lfloor (l-3)/2 \rfloor-4}+\\
		\gamma_{jklm}B_{jk(l-1)0}&z^{-2k-4\lfloor (l-1)/2 \rfloor}+\tilde{P}_{jkl0}(z^{-2})\ .
		\end{split}
		\end{equation*}
		If~$l$ is even then  the leading term is the first one $B_{jkl0}z^{-2k-4\lfloor l/2 \rfloor}$ with
		$B_{jkl0}= \alpha_{jkl0}B_{jk(l-2)1}$. If~$l$ is odd, then the three first terms
		are of the same degree $-2k-4\lfloor l/2 \rfloor$.
		So, we let $B_{jkl0}=\alpha_{jkl0}B_{jk(l-2)1}+\beta_{jkl0}B_{jk(l-3)2}+\gamma_{jkl0}B_{jk(l-1)0}.$ 
		In any case we obtain 
		$$X=(-1)^m B_{jklm}z^{-2k-4\lfloor l/2 \rfloor -2m}+P_{jkl0}(z^{-2})$$
		where $P_{jkl0}$ is of degree smaller than~$k+2\lfloor l/2\rfloor-1$.
		\vspace{1ex}
		%------------------------------------------------------------------------
		\item $m=0, l=2$.
		\begin{equation*}
		\begin{split}
		X=-\alpha_{jklm}z^{-2}(F_j, F_{1+j+k}, F_{2+j+k+0}, F_{3+j+k+0+1})+\\
		\gamma_{jklm}(F_j, F_{1+j+k}, F_{2+j+k+1}, F_{3+j+k+1+0})
		\end{split}
		\end{equation*}
		\begin{equation*}
		\begin{split}
		X=\alpha_{jklm}B_{jk01}&z^{-2k-2-2}+\\
		\gamma_{jklm}B_{jk10}&z^{-2k}+
		\tilde{P}_{jklm}(z^{-2})=\\
		(-1)^m \alpha_{jklm}B_{jk01}& z^{-2k-4\lfloor l/2 \rfloor -2m}+P_{jk20}(z^{-2})
		\end{split}
		\end{equation*}
		where $P_{jk20}$ is of degree smaller than~$k+1$.
		\vspace{1ex}
		%----------------------------------------------------------------------
		\item $m=0, l=1, k\geq 1$.
		\begin{equation*}
		\begin{split}
		X=-\beta_{jklm}&(F_j, F_{1+j+(k-1)}, F_{2+j+(k-1)+0}, F_{3+j+(k-1)+0+1})+\\
		\gamma_{jklm}&(F_j, F_{1+j+k}, F_{2+j+k+0}, F_{3+j+k+0+0})\ .
		\end{split}
		\end{equation*}
		The last expression gives
		\begin{equation*}
		\begin{split}
		X=\beta_{jklm}B_{j(k-1)01}&z^{-2(k-1)-2}+
		\\ \gamma_{jklm}B_{jk00}&z^{-2k}+P_{jklm}(z^{-2})=\\
		(-1)^m (\beta_{jklm}B_{j(k-1)01}+\gamma_{jklm}B_{jk00})&z^{-2k-4\lfloor l/2 \rfloor -2m} +P_{jk10}(z^{-2})
		\end{split}
		\end{equation*}
		with $P_{jk10}$ of degree smaller than~$k-1$.
		\vspace{1ex}
		%---------------------------------------------------------------------
		\item $m=0, l=1, k=0$.
		\begin{equation*}
		X=\gamma_{jklm}(F_j, F_{1+j+0}, F_{2+j+0+0}, F_{3+j+0+0+0}).
		\end{equation*}
		This is simply a positive constant (by induction)
		\begin{equation*}
		\begin{split}
		X=\gamma_{jklm}B_{j000}=(-1)^m \gamma_{jklm}B_{j000}z^{-2k-4\lfloor l/2 \rfloor -2m}\ .
		\end{split}
		\end{equation*}
		%	\vspace{1ex}
		%-------------------------------------------------------------------------
		\item $m=0, l=0, k\geq 2$.
		\begin{equation*}
		\begin{split}
		X=\alpha_{jklm}z^{-2}&(F_j, F_{1+j+(k-1)}, F_{2+j+(k-1)+0}, F_{3+j+(k-1)+0+0})+\\
		-\beta_{jklm}&(F_j, F_{1+j+(k-2)}, F_{2+j+(k-2)+1}, F_{3+j+(k-2)+1+0})\ .
		\end{split}
		\end{equation*}
		Hence,
		\begin{equation*}
		\begin{split}
		X=\alpha_{jklm}B_{j(k-1)00}&z^{-2(k-1)-2}+\\
		-\beta_{jklm}B_{j(k-2)10}&z^{-2(k-2)} +\tilde{P}_{jklm}(z^{-2}) =\\
		(-1)^m \alpha_{jklm}B_{j(k-1)00} &z^{-2k-4\lfloor l/2 \rfloor -2m}+P_{jk00}(z^{-2})
		\end{split}
		\end{equation*}
		where $P_{jk00}$ is of degree smaller than~$k-1$.
		\vspace{1ex}
		%-------------------------------------------------------------------------
		\item $m=0, l=0, k=1$.
		\begin{equation*}
		X=\alpha_{jklm}z^{-2}(F_j, F_{1+j+0}, F_{2+j+0+0}, F_{3+j+0+0+0})
		\end{equation*}
		Thus,
		\begin{equation*}
		X= \alpha_{jklm}B_{j000}z^{-2} =
		(-1)^m \alpha_{jklm}B_{j000}z^{-2k-4\lfloor l/2 \rfloor -2m}\ .
		\end{equation*}
		%--------------------------------------------------------------------------
		\item $m=0, l=0, k=0, j\geq 1$.
		\begin{equation*}
		X=\beta_{jklm}(F_{j-1}, F_{1+(j-1)+0}, F_{2+(j-1)+0+0}, F_{3+(j-1)+0+0+0})\ .
		\end{equation*}
		By induction this is  a constant
		$$X=\beta_{jklm}B_{(j-1)000}=(-1)^m \beta_{jklm}B_{(j-1)000}z^{-2k-4\lfloor l/2 \rfloor -2m}\ .$$
		%------------------------------------------------------------------------
	\end{enumerate}
\end{proof}

%=====================================================
\section{Some elements from Siegel-Shidlovskii Theory - a zero is transcendental}
%=====================================================
We recall the notion of a Siegel $E$-function.
Let~$E$ be a power series.
$$E(z)=\sum_{k=0}^{\infty} a_k \frac{z^k}{k!}\ .$$
\begin{definition}[\cite{sieg-1949}*{Ch.\ II.1}]
	$E$~is called an $E$-function if the following two conditions hold
	\begin{enumerate}[label=\textup{(\alph*)}]
		\item $a_k\in\Qb$ for all~$k$.
		\item If $a_k=p_k/q_k$, where $p_k, q_k\in\Zb$ are coprime, then 
		$a_k=o(k^{\eps k})$, and $q_k=o(k^{\eps k})$ as $k\to\infty$ for all $\eps>0$.
	\end{enumerate}
\end{definition}
We remark that any $E$-function is entire and~$E$ functions constitute a ring.
The examples we are interested in are the functions $J_m, J_m', I_m, I_m'$.
It is readily verified that these are all $E$-functions.

Siegel-Shidlovskii theory is concerned with transcendental properties of values of $E$-functions which satisfy a linear ODE system. 
The following theorem is one of the corner stones of the theory. It was proved for second order ODEs in~\cite{sieg-1929} and~\cite{sieg-1949},
and then it was simplified and extended to linear systems by Shidlovskii.
\begin{theorem}[\citelist{\cite{shid1959}\cite{shid-book}*{Ch.\ 3\S13, Second Fundamental Theorem}}]
	\label{thm:sieg-shid}
	Let $E_1, \ldots, E_k$ be algebraically independent $E$-functions over the field of rational functions $\Cb(z)$.
	Let $E=(E_1, \ldots, E_k)$ satisfy a linear ODE system of the form
	\begin{equation}
	\label{eq:ode}
	E'=A E\ ,
	\end{equation}
	where $A\in M_k(\Cb(z))$.
	Let $\alpha$ be algebraic and distinct from the poles of $A_{ij}$. 
	Then, the set of numbers $\{E_j(\alpha)\}_{j=1}^k$ is algebraically independent over~$\Qb$. 
\end{theorem}

The assumptions in Theorem~\ref{thm:sieg-shid} are verified in the 
case relevant to this paper by an earlier theorem of Siegel. 
\begin{theorem}[\cite{sieg-1929}, see also~\cite{shid-book}*{Ch.\ 9, \S5, Lemma 6}]
	\label{thm:siegel}
	The four $E$-functions
	$J_m, J_m', I_m, I_m'$ are algebraically independent over $\Cb(z)$.
\end{theorem}

As a corollary we have
\begin{corollary}
	\label{cor:sieg-shid}
	Let $x_0>0$. If $W_m(x_0)=0$ then $x_0$ is transcendental.
\end{corollary}
\begin{proof}[Proof of Corollary~\ref{cor:sieg-shid}]
	The vector of functions $E=(J_m, J_m', I_m, I_m')$ satisfies an ODE
	of the form~(\ref{eq:ode})  with 
	$$	A(z)=\begin{pmatrix}
	0 & 1 &0 &0 \\
	-1+\frac{m^2}{z^2} &-\frac{1}{z} &0&0\\
	0&0&0&1\\
	0&0&1+\frac{m^2}{z^2}&-\frac{1}{z}
	\end{pmatrix}$$
	
	Let $\alpha$ be a positive algebraic number.
	By Theorems~\ref{thm:siegel} and~\ref{thm:sieg-shid} the four values
	$J_m(\alpha), J_m'(\alpha), I_m(\alpha), I_m'(\alpha)$ are 
	algebraically independent. In particular, $W_m(\alpha)$ as a polynomial in these numbers is not~$0$. 
\end{proof}

%==================================================
\section{Proof of the main Theorem~\ref{thm:wm-zeros}}
%=================================================
The main theorem now follows easily.

\begin{proof}[Proof of Theorem~\ref{thm:wm-zeros}]
	Let $x_0>0$ be a common zero of $W_{m_0}, W_{m_1}, W_{m_2}$ and~$W_{m_3}$.
	The functions~$W_0, W_1, W_2$ and $W_3$ have no common zero by Theorem~\ref{thm:w0-w1-zero-forbidden}. Hence, we can apply
	proposition~\ref{prop:algebraic} to conclude that the positive number $x_0$ must be algebraic. On the other hand,
	by Corollary~\ref{cor:sieg-shid} it must be transcendental.
	This is a contradiction.
\end{proof}

\begin{remark}
	The full power of Theorem~\ref{thm:w0-w1-zero-forbidden}
	was not used in the proof of Theorem~\ref{thm:wm-zeros}. The weaker statement that $W_0, W_1, W_2$ and~$W_3$ have no common zero follows also by computing a fourth order ODE for~$W_0$
	and showing that $(W_0, W_1, W_2, W_3)$ is obtained
	from $(W_0, W_0', W_0'', W_0''')$ by an invertible transformation.
	We leave the details to the reader.
\end{remark}
%=====================================
\section{Appendix-Shortest recursion possible}
%=====================================
We explain how our arguments for the proof of Theorem~\ref{thm:wm-zeros}
also show that any four distinct~$W_m$'s are algebraically independent over the field of rational functions~$\Cb(z)$. In particular, it follows
that the linear recursion in Theorem~\ref{thm:recursion} cannot be shortened while keeping rational coefficients.

\begin{claim}[base case]
	\label{clm:0123-ind}
	The four functions
	$W_0, W_1, W_2, W_3$ are algebraically independent over the field~$\Cb(z)$.
\end{claim}
\begin{proof}
	We may express the function $W_m$ as a linear combination
	 of the four functions $I_0 J_0, I_0' J_0, I_0 J_0'$ and~$I_0' J_0'$
	 over the field~$\Qb(z)$ simply by expanding the defining formula~\eqref{eqn:Wm-def}
	by means of the classical recursions in Proposition~\ref{prop:bessel-recursions}.
	One calculates
	\begin{equation}
	\label{eq:w0123}
	\begin{pmatrix}
	W_0\\W_1\\W_2\\W_3
	\end{pmatrix}=
	\begin{pmatrix}
	0 & 1 & -1 &0\\
	0&-1&-1&0\\
	0&-1&1&-\frac{4}{z}\\
		-\frac{8}{z} &1+\frac{16}{z^2}&1-\frac{16}{z^2}&\frac{32}{z^3}
	\end{pmatrix}
	\begin{pmatrix}
	I_0 J_0\\ I_0'J_0\\I_0 J_0' \\ I_0' J_0'
	\end{pmatrix}\ .
	\end{equation}
	By Theorem~\ref{thm:siegel} the four functions~$I_0 J_0$, $I_0' J_0$, $I_0 J_0'$,
	$I_0' J_0'$ are algebraically independent over the field~$\Cb(z)$.
	Since the linear system~(\ref{eq:w0123}) is invertible and due to the simple fact that the set of non-zero polynomials is preserved by linear transformations we
	obtain that $W_0$, $W_1$, $W_2$, $W_3$ are algebraically independent over~$\Cb(z)$ too.
\end{proof}
\begin{theorem}
	\label{thm:lin-ind}
	Let $m_0, m_1, m_2, m_3$ be four distinct non-negative integers.
	Then, $W_{m_0}, W_{m_1}, W_{m_2}$ and~$W_{m_3}$ are algebraically independent over the field~$\Cb(z)$.
\end{theorem}
\begin{proof}
By equation~\eqref{eq:f_m-expansion} we can express
$$W_{m_j}= A_{j0} W_0+A_{j1}W_1+A_{j2}W_{j2}+A_{j3}W_3\ ,$$
with $A\in \mathrm{GL}_4(\Cb(z))$. By Claim~\ref{clm:0123-ind}  it follows that~$W_{m_0}, W_{m_1}, W_{m_2}$ and~$W_{m_3}$ are algebraically
independent.
\end{proof}

\end{document}